\documentclass[11pt]{article}
\usepackage{authblk}
\usepackage{amsmath,amssymb,mathrsfs}
\usepackage{enumitem}
\usepackage{verbatim}
\usepackage{color}
\usepackage[margin=1in]{geometry}
\usepackage{setspace}

\newcommand{\lang}{\mathcal{L}}
\newcommand{\limp}{\longrightarrow}
\newcommand{\pow}{\mathcal{P}}
\newcommand{\var}{\mathrm{var}}

\newcommand{\lrri}{\vdash^{re}}

\usepackage{amsthm}
\theoremstyle{definition}
\newtheorem{theorem}{Theorem}[section]
\newtheorem{lemma}[theorem]{Lemma}
\newtheorem{corollary}[theorem]{Corollary}
\newtheorem{definition}[theorem]{Definition}
\newtheorem{remark}[theorem]{Remark}
\newtheorem{example}[theorem]{Example}

\usepackage[backref=page,colorlinks=true,allcolors=blue]{hyperref}

  \title{Relational Companions of Logics\footnote{A version of this article has been submitted to the Indian Conference on Logic and its Applications 2025}}
  \author[1]{Sankha S. Basu}
  \author[1]{Sayantan Roy}
  \date{August 30, 2024}
  \affil[1]{Department of Mathematics\\
  Indraprastha Institute of Information Technology-Delhi\\
  New Delhi, India.}
  
\begin{document}
\maketitle  
\begin{abstract}
\noindent The variable inclusion companions of logics have lately been thoroughly studied by multiple authors. There are broadly two types of these companions: the left and the right variable inclusion companions. Another type of companions of logics induced by Hilbert-style presentations (Hilbert-style logics) were introduced in \cite{BasuChakraborty2021}. A sufficient condition for the restricted rules companion of a Hilbert-style logic to coincide with its left variable inclusion companion was proved there, while a necessary condition remained elusive. The present article has two parts. In the first part, we give a necessary and sufficient condition for the left variable inclusion and the restricted rules companions of a Hilbert-style logic to coincide. In the rest of the paper, we recognize that the variable inclusion restrictions used to define variable inclusion companions of a logic $\langle\lang,\vdash\rangle$ are relations from $\pow(\lang)$ to $\lang$. This leads to a more general idea of a relational companion of a logical structure, a framework that we borrow from the field of universal logic. We end by showing that even Hilbert-style logics and the restricted rules companions of these can be brought under the umbrella of the general notions of logical structures and their relational companions that are discussed here.
\end{abstract}

\textbf{Keywords:} 
Companion logics; Universal logic; Logics of variable inclusion.

\tableofcontents

\section{Introduction}
The \emph{logics of variable inclusion} have recently been rigorously studied, e.g., in \cite{Bonzio2020,BonzioPaoliPraBaldi2022,PaoliPraBaldiSzmuc2021}. These companion logics come in four flavors, viz., the \emph{left}, the \emph{right}, the \emph{pure left}, and the \emph{pure right variable inclusion companion} logics. The definitions and various examples of each of these classes can be found in the above references. It is well-known that the left variable inclusion companion of classical propositional logic (CPC) is the paraconsistent weak Kleene logic (PWK). A simple Hilbert-style presentation of PWK, consisting of the same set of axioms as CPC and a restricted version of the classical rule of modus ponens, was presented in \cite{Bonzio2017}. This led to the following natural question. Can we always obtain a Hilbert-style presentation of the left variable inclusion companion logic from that of the original logic (if it has one, of course) by just restricting the rules of inference? The answer to this question was shown to be negative in \cite{BasuChakraborty2021}. In the course of the argument, the \emph{restricted rules companion} of a Hilbert-style logic, i.e., a logic induced syntactically by a Hilbert-style presentation, was introduced. A sufficient condition for the restricted rules companion to coincide with the left variable inclusion companion of a Hilbert-style logic was proved as well. However, a necessary condition for the same remained unattained. 

The present article has two main parts. In the first, a necessary and sufficient condition for the left variable inclusion and the restricted rules companions of a Hilbert-style logic to coincide, is presented. 

In the second part, we generalize the notions of variable inclusion companion and restricted rules companion logics. This is done using the framework of \emph{logical structures} from universal logic \cite{Beziau1994,Beziau2006}. 

A logical structure is a pair $\mathcal{S}=\langle\lang,\vdash\rangle$, where $\lang$ is a set and $\vdash\,\subseteq\pow(\lang)\times\lang$. A \emph{logic} is a logical structure $\mathcal{S}=\langle\lang,\vdash\rangle$, where $\lang$ is the set of formulas defined inductively, in the usual way, over a set of variables $V$, using a finite set of connectives/operators called the \emph{signature/type}. In other words, $\lang$ is the formula algebra over $V$ of some type. The formula algebra has the universal mapping property for the class of all algebras of the same type as $\lang$ over $V$, i.e., any function $f:V\to A$, where $A$ is the universe of an algebra $\mathbf{A}$ of the same type as $\mathcal{L}$, can be uniquely extended to a homomorphism from $\lang$ to $\mathbf{A}$ (see \cite{FontJansanaPigozzi2003,Font2016} for more details). We, however, do not assume any condition on the $\vdash$-relation. For any $\alpha\in\lang$, $\var(\alpha)$ denotes the set of all the variables occurring in $\alpha$, and for any $\Delta\subseteq\lang$, $\var(\Delta)=\displaystyle\bigcup_{\alpha\in\Delta}\var(\alpha)$.

\section{Left variable inclusion and restricted rules companions}

In this section, we will deal exclusively with logics, in the usual sense, as described in the previous section. As mentioned earlier, the \emph{logics of variable inclusion} have recently been rigorously studied, e.g., in \cite{Bonzio2020,BonzioPaoliPraBaldi2022,PaoliPraBaldiSzmuc2021}. The following definition of a left variable inclusion companion can be found in these. 

\begin{definition}\label{dfn:lvarinclusion}
Suppose $\mathcal{S}=\langle\lang,\vdash\rangle$ is a logic. The \emph{left variable inclusion companion} of $\mathcal{S}$, denoted by $\mathcal{S}^{l}=\langle\lang,\vdash^{l}\rangle$, is defined as follows. For any $\Gamma\cup\{\alpha\}\subseteq\lang$, 
    \[
    \Gamma\vdash^{l}\alpha\;\hbox{iff there is a}\;\Delta\subseteq\Gamma\;\hbox{such that}\;\var(\Delta)\subseteq\var(\alpha)\;\hbox{and}\;\Delta\vdash\alpha.
    \]
\end{definition}

The restricted rules companion of a Hilbert-style logic, i.e., a logic induced syntactically by a Hilbert-style presentation, was introduced in \cite{BasuChakraborty2021} as follows.

\begin{definition}\label{dfn:res_rules_comp}
Suppose $\mathcal{S}=\langle\lang,\vdash\rangle$ is a Hilbert-style logic with $A\subseteq\lang$ as the set of axioms and $R_\mathcal{S}\subseteq\pow(\lang)\times\lang$ as the set of rules of inference. The \emph{restricted rules companion} of $\mathcal{S}$, denoted by $\mathcal{S}^{re}=\langle\lang,\lrri\rangle$, is then defined as the Hilbert-style logic with the following sets of axioms and rules.

\begin{tabular}{lcl}
       Set of axioms&=&$A$, and\\
       set of rules of inference&=&$R_{\mathcal{S}^{re}}=\left\{\dfrac{\Gamma}{\alpha}\in R_\mathcal{S}\mid\,\var(\Gamma)\subseteq\var(\alpha)\right\}$.
  \end{tabular}
\end{definition}

Suppose $\mathcal{S}=\langle\lang,\vdash\rangle$ is a Hilbert-style logic and 
$\mathcal{S}^l=\langle\lang,\vdash^l\rangle$, $\mathcal{S}^{re}=\langle\lang,\lrri\rangle$ be its left variable inclusion and restricted rules companion logics, respectively. It was established that, $\lrri\,\subseteq\,\vdash^l$ (\cite[Theorem 3.6]{BasuChakraborty2021}), but the reverse inclusion, i.e., $\vdash^l\,\subseteq\,\lrri$ is not guaranteed  (\cite[Remark 3.7]{BasuChakraborty2021}). The following sufficient condition for the latter inclusion was also given in this paper.

\begin{theorem}[{\cite[Theorem 4.3]{BasuChakraborty2021}}]
Suppose $\mathcal{S}=\langle\lang,\vdash\rangle$ is a finitary Hilbert-style logic such that $\dfrac{\alpha,\alpha\limp\beta}{\beta}$ (modus ponens [MP]) is a rule of inference in $\mathcal{S}$. Suppose further that the Deduction theorem holds in $\mathcal{S}$. Then the restricted rules companion of $\mathcal{S}$ coincides with the left variable inclusion companion of $\mathcal{S}$, i.e., $\lrri\,=\,\vdash^l$. 
\end{theorem}

However, a necessary condition for $\vdash^l\,=\,\lrri$ remained elusive. In the rest of this section, we investigate this further and provide a necessary and sufficient condition for the two companion logics to coincide. The following lemmas list some straightforward inferences that can be drawn from the definitions of left variable inclusion and restricted rules companions.

\begin{lemma}\label{lem:Sl-S,Sre-S}
Suppose $\mathcal{S}=\langle\lang,\vdash\rangle$ is a logic.
\begin{enumerate}[label=(\roman*)]
    \item $\mathcal{S}^l=\langle\lang,\vdash^l\rangle$ is monotonic. Moreover, if $\mathcal{S}$ is monotonic, then $\vdash^l\,\subseteq\,\vdash$.
    \item $(\mathcal{S}^l)^l=\mathcal{S}^l$, i.e., $(\vdash^l)^l\,=\,\vdash^l$.
    \item Suppose $\mathcal{S}$ is a Hilbert-style logic. Then, $\mathcal{S}^{re}$ is monotonic and moreover, $\lrri\,\subseteq\,\vdash$.
    \item Suppose $\mathcal{S}$ is a Hilbert-style logic. Then, $(\mathcal{S}^{re})^{re}=\mathcal{S}^{re}$, i.e., $(\lrri)^{re}\,=\,\lrri$.
\end{enumerate}
\end{lemma}

\begin{proof} 
Parts (i) and (ii) follow straightforwardly from the definition of left variable inclusion companions. For part (iii), we recall that every Hilbert-style logic is monotonic. Thus, in fact, $\mathcal{S}$ and $\mathcal{S}^{re}$, being Hilbert-style logics, are both monotonic. That $\vdash^{re}\,\subseteq\,\vdash$, follows from the definition of $\mathcal{S}^{re}$.

For part (iv), let $A$ and $R_\mathcal{S}$ as its sets of axioms and rules of inference, respectively, of $\mathcal{S}$. Then $\mathcal{S}^{re}$ has $A$ and $R_{\mathcal{S}^{re}}$, as its sets of axioms and rules, where $R_{\mathcal{S}^{re}}$ is as described in the definition of a restricted rules companion. Now, since any rule in $R_{\mathcal{S}^{re}}$ is already restricted, $A$ and $R_{\mathcal{S}^{re}}$ also comprise a Hilbert-style presentation for $(\mathcal{S}^{re})^{re}$. Thus, $(\mathcal{S}^{re})^{re}=\mathcal{S}^{re}$, i.e., $(\lrri)^{re}\,=\,\lrri$.
\end{proof}

\begin{theorem}\label{thm:l1-l2,re1-re2}
    Suppose $\mathcal{S}_1=\langle\lang,\vdash_1\rangle$ and $\mathcal{S}_2=\langle\lang,\vdash_2\rangle$ are two logics such that $\vdash_1\,\subseteq\,\vdash_2$. Then, $\vdash_1^l\,\subseteq\,\vdash_2^l$. Moreover, if $\mathcal{S}_1$ and $\mathcal{S}_2$ are Hilbert-style logics, such that for any $\dfrac{\Gamma}{\alpha}\in R_{\mathcal{S}_1^{re}}$, $\Gamma\lrri_2\alpha$, then $\lrri_1\,\subseteq\,\lrri_2$.
\end{theorem}

\begin{proof}
    Suppose $\Sigma\cup\{\alpha\}\subseteq\lang$ such that $\Sigma\vdash_1^l\alpha$. Then there exists $\Delta\subseteq\Sigma$ with $\var(\Delta)\subseteq\var(\alpha)$ such that $\Delta\vdash_1\alpha$. Since $\vdash_1\,\subseteq\,\vdash_2$, $\Delta\vdash_2\alpha$. Thus, $\Sigma\vdash_2^l\alpha$. Hence, $\vdash_1^l\,\subseteq\,\vdash_2^l$.

    Next, suppose $\mathcal{S}_1,\mathcal{S}_2$ are Hilbert-style logics such that for any $\dfrac{\Gamma}{\alpha}\in R_{\mathcal{S}_1^{re}}$, $\Gamma\lrri_2\alpha$. Now, let $\Sigma\cup\{\alpha\}\subseteq\lang$ such that $\Sigma\lrri_1\alpha$. Then, there exists a derivation $D=\langle\alpha_1,\ldots,\alpha_n=\alpha\rangle$ of $\alpha$ from $\Sigma$, where for each $1\le i\le n$, $\alpha_i$ is either an axiom of $\mathcal{S}_1$, or an element of $\Sigma$, or is obtained by applying a rule of inference in $R_{\mathcal{S}_1^{re}}$. If $\alpha_i$ is an axiom, then $\vdash_1\alpha_i$, and since $\vdash_1\,\subseteq\,\vdash_2$, $\vdash_2\alpha_i$. Then, by \cite[Theorem 3.4]{BasuChakraborty2021}, $\lrri_2\alpha_i$. So, there exists a derivation of $\alpha_i$ in $\mathcal{S}_2^{re}$. Now, suppose $\alpha_i$ is obtained by applying a rule of inference $\dfrac{\Delta}{\alpha_i}\in R_{\mathcal{S}_1^{re}}$, where $\Delta\subseteq\{\alpha_1,\ldots,\alpha_{i-1}\}$. Then, by assumption, $\Delta\lrri_2\alpha_i$, and hence, there exists a derivation of $\alpha_i$ from $\Delta$ in $\mathcal{S}_2^{re}$. So, we can translate the derivation $D$ of $\alpha$ from $\Sigma$ as follows. For each $1\le i\le n$, if $\alpha_i$ is an axiom, then we replace $\alpha_i$ by the elements of a derivation of $\alpha_i$ in $\mathcal{S}_2^{re}$. If $\alpha_i$ is obtained from $\Delta\subseteq\{\alpha_1,\ldots,\alpha_{i-1}\}$, by using a rule of inference in $\mathcal{S}_1^{re}$, then we replace $\alpha_i$ by the elements of a derivation of $\alpha_i$ from $\Delta$ in $\mathcal{S}_2^{re}$. Finally, if $\alpha_i\in\Sigma$, then we keep it unchanged. Let $D^\prime$ be the resulting sequence of formulas. Clearly, $D^\prime$ is a derivation of $\alpha$ from $\Sigma$ in $\mathcal{S}_2^{re}$. Thus, $\Sigma\lrri_2\alpha$. Hence, $\lrri_1\,\subseteq\,\lrri_2$.
\end{proof}

\begin{remark}
    It is not true, in general, that if $\mathcal{S}_1=\langle\lang,\vdash_1\rangle$ and $\mathcal{S}_2=\langle\lang,\vdash_2\rangle$ are two Hilbert-style logics with $\vdash_1\,\subseteq\,\vdash_2$, then $\lrri_1\,\subseteq\,\lrri_2$. This can be seen from the following example.

    Suppose $\lang$ is the formula algebra over a countable set of variables $V$ of type $\{\land,\lor\}$. Let $\mathcal{S}_1=\langle\lang,\vdash_1\rangle$ be the Hilbert-style logic with an empty set of axioms and the following two rules of inference.
    \[
    R_1:\,\dfrac{\alpha\land\beta}{\alpha}\quad\hbox{and}\quad R_2:\,\dfrac{\alpha}{\alpha\lor\beta},\quad\hbox{where }\alpha,\beta\in\lang.
    \]
    $\mathcal{S}_1$ is the same logic that was used in \cite[Remark 3.7]{BasuChakraborty2021} to show that the left variable inclusion companion of a Hilbert-style logic can differ from its restricted rules companion.
    
    Let $\mathcal{S}_2=\langle\lang,\vdash_2\rangle$ be the Hilbert-style logic with an empty set of axioms and the following rule of inference in addition to $R_1, R_2$ above.
    \[
    R_3:\,\dfrac{\alpha\land\beta}{\alpha\lor\beta},\quad\hbox{where }\alpha,\beta\in\lang.
    \]
    Clearly, $\vdash_2\,\subseteq\,\vdash_1$, since $R_3$ can be derived in $\mathcal{S}_1$ as follows. For any $\alpha,\beta\in\lang$,
    \[
    \begin{array}{lcl}
         \alpha\land\beta&\vdash_1&1.\, \alpha\land\beta\\
         &&2.\,\alpha\qquad [R_1\hbox{ on (1)}]\\
         &&3.\,\alpha\lor\beta\qquad [R_2\hbox{ on (2)}]
    \end{array}
    \]
    We note that $\mathcal{S}_1^{re}=\langle\lang,\lrri_1\rangle$, is the logic induced by the same set of axioms and the following two rules of inference.
    \[
    \begin{array}{ll}
         R_1^\prime:\,\dfrac{\alpha\land\beta}{\alpha}&\hbox{such that }\var(\alpha\land\beta)\subseteq\var(\alpha),\hbox{ i.e., }\var(\beta)\subseteq\var(\alpha),\hbox{ and}\\
         &\\
         R_2:\,\dfrac{\alpha}{\alpha\lor\beta},&
    \end{array}
    \]
    while $\mathcal{S}_2^{re}=\langle\lang,\lrri_1\rangle$, is the logic induced by the same set of axioms and the rules $R_1^\prime,R_2$, and $R_3$. ($R_2$ and $R_3$ do not need to be restricted as  $\var(\alpha)\subseteq\var(\alpha\lor\beta)$ and $\var(\alpha\land\beta)=\var(\alpha\lor\beta)$ for any $\alpha,\beta\in\lang$.)
    
    Now, suppose $p,q$ are distinct variables. Then, while $p\land q\lrri_2 p\lor q$, $p\land q\not\lrri_1 p\lor q$, since we cannot apply $R_1^\prime$ and $R_2$ to derive $p\lor q$ from $p\land q$ in $\mathcal{S}_1$, as shown in \cite[Remark 3.7]{BasuChakraborty2021}.

    Thus, $\lrri_2\,\not\subseteq\,\lrri_1$ although $\vdash_2\,\subseteq\,\vdash_1$.
\end{remark}

\begin{theorem}
    Suppose $\mathcal{S}=\langle\lang,\vdash\rangle$ is a Hilbert-style logic. Then, $\mathcal{S}^l=\mathcal{S}^{re}$ iff $(\mathcal{S}^{re})^l=\mathcal{S}^l$. In other words, $\vdash^l\,=\,\lrri$ iff $(\lrri)^l\,=\,\vdash^l$, i.e., for all $\Gamma\cup\{\alpha\}\subseteq\lang$, $\Gamma\vdash^l\alpha$ iff there exists $\Delta\subseteq\Gamma$ such that $\var(\Delta)\subseteq\var(\alpha)$ and $\Delta\lrri\alpha$. 
\end{theorem}

\begin{proof}
    Suppose $\mathcal{S}^l=\mathcal{S}^{re}$, i.e., $\vdash^l\,=\,\lrri$. Now, by Lemma \ref{lem:Sl-S,Sre-S} (iii), $\lrri\,\subseteq\,\vdash$. So, by Theorem \ref{thm:l1-l2,re1-re2}, $(\lrri)^l\,\subseteq\,\vdash^l$. Now, suppose $\Gamma\cup\{\alpha\}\subseteq\lang$ such that $\Gamma\vdash^l\alpha$. Then, there exists $\Delta\subseteq\Gamma$ such that $\var(\Delta)\subseteq\var(\alpha)$ and $\Delta\vdash\alpha$. Clearly, $\Delta\vdash^l\alpha$ as well. This implies that $\Delta\lrri\alpha$, since $\vdash^l\,=\,\lrri$. Now, as $\Delta\subseteq\Gamma$, $\var(\Delta)\subseteq\var(\alpha)$, and $\Delta\lrri\alpha$, $\Gamma(\lrri)^l\alpha$. Thus, $\vdash^l\,\subseteq\,(\lrri)^l$. Hence, $(\lrri)^l\,=\,\vdash^l$.

    Conversely, suppose $(\lrri)^l\,=\,\vdash^l$. We know that $\lrri\,\subseteq\,\vdash^l$. Thus, we only need to show that $\vdash^l\,\subseteq\,\lrri$. Suppose $\Gamma\cup\{\alpha\}\subseteq\lang$ such that $\Gamma\vdash^l\alpha$. Then, by our assumption, $\Gamma(\lrri)^l\alpha$. So, there exists $\Delta\subseteq\Gamma$ such that $\var(\Delta)\subseteq\var(\alpha)$ and $\Delta\lrri\alpha$. Now, as $\mathcal{S}^{re}$ is a Hilbert-style logic, it is monotonic. This implies that $\Gamma\lrri\alpha$. Thus, $\vdash^l\,\subseteq\,\lrri$. Hence $\vdash^l\,=\,\lrri$.
\end{proof}

\section{Relational companions of a logical structure}

We now let go of the formula algebras and land in the arena of logical structures that were described in the introduction. The attempt here would be to generalize the notion of a logic of variable inclusion. In doing so, we will be able to capture a lot more than just the logics of left variable inclusion.

\begin{definition}
    Suppose $\mathcal{S}=\langle\lang,\vdash\rangle$ is a logical structure and $\varrho\subseteq\pow(\lang)\times\lang$. 
    \begin{enumerate}[label=(\roman*)]
        \item The \emph{$\varrho$-companion} of $\mathcal{S}$ is the logical structure $\mathcal{S}^\varrho=
        \langle\lang,\vdash^\varrho\rangle$, where $\vdash^\varrho\,\subseteq\pow(\lang)\times\lang$ is defined as follows. For any $\Gamma\cup\{\alpha\}\subseteq\lang$,
        \[
        \Gamma\vdash^\varrho\alpha\hbox{ iff there is a }\Delta\subseteq\Gamma\hbox{ such that }(\Delta,\alpha)\in\varrho\hbox{ and }\Delta\vdash\alpha.
        \]
        \item The \emph{pure $\varrho$-companion} of $\mathcal{S}$ is the logical structure $\mathcal{S}^{p\varrho}=
        \langle\lang,\vdash^{p\varrho}\rangle$, where $\vdash^{p\varrho}\,\subseteq\pow(\lang)\times\lang$ is defined as follows. For any $\Gamma\cup\{\alpha\}\subseteq\lang$,
        \[
        \Gamma\vdash^{p\varrho}\alpha\hbox{ iff there is a }\Delta\subseteq\Gamma\hbox{ such that }\Delta\neq\emptyset,(\Delta,\alpha)\in\varrho\hbox{ and }\Delta\vdash\alpha.
        \]
    \end{enumerate}
    
    Any such $\varrho$-companion ($p\varrho$-companion) $\mathcal{S}^\varrho$ ($\mathcal{S}^{p\varrho}$) is called a \emph{relational companion} (\emph{pure relational companion}) of $\mathcal{S}$.
\end{definition}

\begin{example}
    Suppose $\mathcal{S}=\langle\lang,\vdash\rangle$ is a logic (in the usual sense).
    \begin{enumerate}[label=(\roman*)]
        \item The left variable inclusion companion $\mathcal{S}^l$ is the relational companion $\mathcal{S}^L=\langle\lang,\vdash^L\rangle$ of $\mathcal{S}$, where the relation $L\subseteq\pow(\lang)\times\lang$ is defined by variable inclusion from left to right, i.e., $(\Delta,\alpha)\in L$ iff $\var(\Delta)\subseteq\var(\alpha)$. Thus, $\mathcal{S}^l=\mathcal{S}^L$, the $L$-companion of $\mathcal{S}$.
        \item The \emph{pure left variable inclusion companion} of $\mathcal{S}$, denoted by $\mathcal{S}^{pl}=\langle\lang,\vdash^{pl}\rangle$, is defined (e.g., in \cite{BonzioPaoliPraBaldi2022,PaoliPraBaldiSzmuc2021}) as follows. For any $\Gamma\cup\{\alpha\}\subseteq\lang$,
        \[
        \Gamma\vdash^{pl}\alpha\;\hbox{iff there is a}\;\Delta\subseteq\Gamma\;\hbox{such that}\;\Delta\neq\emptyset,\var(\Delta)\subseteq\var(\alpha),\;\hbox{and}\;\Delta\vdash\alpha.
        \]
        Clearly, $\mathcal{S}^{pl}$ is the pure $L$-companion of $\mathcal{S}$, where the relation $L$ is as defined in (i), i.e., $\mathcal{S}^{pl}=\mathcal{S}^{pL}$.
    \end{enumerate}
\end{example}

\begin{example}
Suppose $\mathcal{S}=\langle\lang,\vdash\rangle$ is a monotonic logic (in the usual sense). 
    \begin{enumerate}[label=(\roman*)]
        \item The right variable inclusion companion of $\mathcal{S}$, denoted by $\mathcal{S}^r=\langle\lang,\vdash^r\rangle$ is defined (e.g., in \cite{PaoliPraBaldiSzmuc2021}) as follows. For any $\Gamma\cup\{\alpha\}\subseteq\lang$, 
        \[
        \Gamma\vdash^{r}\alpha\;\hbox{iff }\;\Gamma\;\hbox{contains an \emph{$\mathcal{S}$-antitheorem}, or}\;\Gamma\vdash\alpha\;\hbox{and}\;\var(\alpha)\subseteq\var(\Gamma).
        \]
        An $\mathcal{S}$-antitheorem is a set $\Sigma\subseteq\lang$ such that for every substitution $\sigma$, $\sigma(\Sigma)\vdash\varphi$ for all $\varphi\in\lang$, i.e., $\sigma(\Sigma)$ \emph{explodes} in $\mathcal{S}$ for every $\sigma$\footnote{Suppose $\lang$ is a formula algebra over a set of variables $V$ (which is the case when discussing variable inclusion logics). A \emph{substitution} is any function $\sigma:V\to\lang$ that extends to a unique endomorphism (also denoted by $\sigma$) from $\lang$ to itself via the universal mapping property.}. 

        $\mathcal{S}^r$ can be seen as the relational companion $\mathcal{S}^R=\langle\lang,\vdash^R\rangle$ of $\mathcal{S}$, where the relation $R\subseteq\pow(\lang)\times\lang$ is defined as follows. $(\Delta,\alpha)\in R$ iff either $\Delta$ is an $\mathcal{S}$-antitheorem or $\var(\alpha)\subseteq\var(\Delta)$.

        The proof that $\mathcal{S}^r=\mathcal{S}^R$ can be given as follows. Let $\Gamma\cup\{\alpha\}\subseteq\lang$.

        Suppose $\Gamma\vdash^R\alpha$. Then there exists a $\Delta\subseteq\Gamma$ such that $(\Delta,\alpha)\in R$ and $\Delta\vdash\alpha$. Now, since $(\Delta,\alpha)\in R$, either $\Delta$ is an $\mathcal{S}$-antitheorem, in which case $\Gamma$ contains an $\mathcal{S}$-antitheorem, or $\var(\alpha)\subseteq\var(\Delta)$ and $\Delta\vdash\alpha$. In the latter case, as $\Delta\subseteq\Gamma$, $\var(\alpha)\subseteq\var(\Gamma)$ and since $\mathcal{S}$ is monotonic, $\Gamma\vdash\alpha$ as well. Thus, $\Gamma\vdash^r\alpha$.

        Conversely, suppose $\Gamma\vdash^r\alpha$. Then, either $\Gamma$ contains an $\mathcal{S}$-antitheorem or $\var(\alpha)\subseteq\var(\Gamma)$ and $\Gamma\vdash\alpha$. If $\Gamma$ contains an $\mathcal{S}$-antitheorem $\Delta$, then $\Delta\subseteq\Gamma$ such that $(\Delta,\alpha)\in R$ and $\Delta\vdash\alpha$ as $\Delta$ is an $\mathcal{S}$-antitheorem. On the other hand, if $\Gamma\vdash\alpha$ and $\var(\alpha)\subseteq\var(\Gamma)$, then $\Gamma$ is its own subset such that $(\Gamma,\alpha)\in R$ and $\Gamma\vdash\alpha$. Thus, in either case, there exists $\Delta\subseteq\Gamma$ such that $(\Delta,\alpha)\in R$ and $\Delta\vdash\alpha$, and so, $\Gamma\vdash^R\alpha$. Hence $\vdash^r\,=\,\vdash^R$, i.e., $\mathcal{S}^r=\mathcal{S}^R$.
        
        \item The \emph{pure right variable inclusion companion} of $\mathcal{S}$, denoted by $\mathcal{S}^{pr}=\langle\lang\vdash^{pr}\rangle$, is defined (e.g., in \cite{BonzioPaoliPraBaldi2022,PaoliPraBaldiSzmuc2021}) as follows. For any $\Gamma\cup\{\alpha\}\subseteq\lang$, 
        \[    \Gamma\vdash^{pr}\alpha\;\hbox{iff}\;\Gamma\vdash\alpha\;\hbox{and}\;\var(\alpha)\subseteq\var(\Gamma).
        \]

        $\mathcal{S}^{pr}$ can be seen as a relational companion $\mathcal{S}^{PR}=\langle\lang,\vdash^{PR}\rangle$ of $\mathcal{S}$, where the relation $PR\subseteq\pow(\lang)\times\lang$ is defined by variable inclusion from right to left, i.e., $(\Delta,\alpha)\in PR$ iff $\var(\alpha)\subseteq\var(\Delta)$. This can be shown with a similar argument as for the right variable inclusion companions above. The monotonicity of $\mathcal{S}$ is required for showing that $\vdash^{PR}\,\subseteq\,\vdash^{pr}$. 

        Although $\mathcal{S}^{pr}$ can be seen as a relational companion of $\mathcal{S}$, it cannot be described as a pure relational companion of $\mathcal{S}$, unless $\mathcal{S}$ is a logic without antitheorems. However, if $\mathcal{S}$ is a logic without antitheorems, then $\mathcal{S}^{pr}=\mathcal{S}^{pR}$, where $R$ is the relation used in (i).
    \end{enumerate}
\end{example}

\begin{example}
    Some more inclusion logics have been introduced in \cite{CaleiroMarcelinoFilipe2020} as generalizations of the left and right variable inclusion logics. In these companion logics, the containment requirement is extended to classes of subformulas. We might call these the \emph{left} and \emph{right subformula inclusion companion logics} (the actual names require some additional machinery and hence we avoid these). While the left subformula inclusion companions of a logic can be seen as relational companions of it, the right subformula inclusion companions of a monotonic logic can be seen as relational companions of it. This is much like the case for left and right variable inclusion companions.
\end{example}

It is easy to see that the left, pure left, and pure right variable inclusion companions of a logic $\mathcal{S}=\langle\lang,\vdash\rangle$, with a unary (negation) operator $\neg$ in the signature, are all paraconsistent with respect to $\neg$, i.e., there exists $\alpha,\beta\in\lang$ such that $\{\alpha,\neg\alpha\}\not\vdash\beta$. In other words, the principle of explosion, ECQ, with respect to $\neg$ (we call this $\neg$-ECQ), fails in these companion logics (see \cite{BasuRoy2023} for more on the paraconsistency of these companion logics). The right variable inclusion companion of a logic is, however, not necessarily paraconsistent with respect to a $\neg$. Thus, not all relational companions of a logic, with such a unary operator $\neg$, are paraconsistent with respect to $\neg$. It can also be observed from the definition of a relational companion that the matter of paraconsistency depends on the relation. The following example from \cite{deSouzaCosta-LeiteDias2021} can be seen as a relational companion created with the intention of obtaining a paraconsistent logical structure. The authors have called this process \emph{paraconsistentization}.

\begin{example}
    Suppose $\mathcal{S}=\langle\lang,\vdash\rangle$ is a logical structure. Then $\mathcal{S}^\mathbb{P}=\langle\lang,\vdash^\mathbb{P}\rangle$ is the logical structure, where $\vdash^\mathbb{P}\subseteq\pow(\lang)\times\lang$ is defined as follows. For any $\Gamma\cup\{\alpha\}\subseteq\lang$,
    \[
    \Gamma\vdash^\mathbb{P}\alpha\hbox{ iff there is a }\Delta\subseteq\Gamma\hbox{ such that }\Delta\hbox{ is }\mathcal{S}\hbox{-nontrivial and }\Delta\vdash\alpha
.    \]
A set $\Delta$ is $\mathcal{S}$-nontrivial, if there exists $\varphi\in\lang$ such that $\Delta\not\vdash\varphi$.
\end{example}

\begin{theorem}\label{thm:SingleRelComp}
    Suppose $\mathcal{S}=\langle\lang,\vdash\rangle$ is a logical structure and $\varrho\subseteq\pow(\lang)\times\lang$. 
    \begin{enumerate}[label=(\roman*)]
        \item Then, $\mathcal{S}^\varrho=\langle\lang,\vdash^\varrho\rangle$ and $\mathcal{S}^{p\varrho}=\langle\lang,\vdash^{p\varrho}\rangle$ are monotonic.
        \item If $\mathcal{S}$ is monotonic, then $\vdash^\varrho,\vdash^{p\varrho}\,\subseteq\,\vdash$.
        \item $(\mathcal{S}^\varrho)^\varrho=\mathcal{S}^\varrho$, i.e., $(\vdash^\varrho)^\varrho\,=\,\vdash^\varrho$, and $(\mathcal{S}^{p\varrho})^{p\varrho}=\mathcal{S}^{p\varrho}$, i.e., $(\vdash^{p\varrho})^{p\varrho}\,=\,\vdash^{p\varrho}$.
    \end{enumerate}
\end{theorem}

\begin{proof}
\begin{enumerate}[label=(\roman*)]
    \item Suppose $\Gamma\cup\Sigma\cup\{\alpha\}\subseteq\lang$ such that $\Gamma\subseteq\Sigma$ and $\Gamma\vdash^\varrho\alpha$. Then, there exists $\Delta\subseteq\Gamma$ such that $(\Delta,\alpha)\in\varrho$ and $\Delta\vdash\alpha$. So, $\Delta\subseteq\Sigma$ as well. Thus, $\Sigma\vdash^\varrho\alpha$. Hence, $\mathcal{S}^\varrho$ is monotonic. 

    An almost identical argument can be used to show that $\mathcal{S}^{p\varrho}$ is monotonic.

    \item Suppose $\mathcal{S}$ is monotonic. Let $\Gamma\cup\{\alpha\}\subseteq\lang$ and $\Gamma\vdash^\varrho\alpha$. Then, there exists $\Delta\subseteq\Gamma$ such that $(\Delta,\alpha)\in\varrho$ and $\Delta\vdash\alpha$. So, by monotonicity, $\Gamma\vdash\alpha$. Hence, $\vdash^\varrho\,\subseteq\,\vdash$. It is easy to see that $\vdash^{p\varrho}\,\subseteq\,\vdash^\varrho$. Thus, if $\vdash^\varrho\,\subseteq\,\vdash$, then $\vdash^{p\varrho}\,\subseteq\,\vdash$ as well.

    \item By (i), $\mathcal{S}^\varrho$ is monotonic. So, by (ii), $(\vdash^\varrho)^\varrho\,\subseteq\,\vdash^\varrho$.
    
    Suppose $\Gamma\cup\{\alpha\}\subseteq\lang$ and $\Gamma\vdash^\varrho\alpha$. Then, there exists $\Delta\subseteq\Gamma$ such that $(\Delta,\alpha)\in\varrho$ and $\Delta\vdash\alpha$. Clearly, $\Delta\vdash^\varrho\alpha$ as well. Thus, $\Gamma(\vdash^\varrho)^\varrho\alpha$, and so, $\vdash^\varrho\,\subseteq\,(\vdash^\varrho)^\varrho$. Hence, $(\vdash^\varrho)^\varrho\,=\,\vdash^\varrho$.

    The argument for $\mathcal{S}^{p\varrho}$ is identical, except for the non-emptiness requirement on the $\Delta$ in the converse part.
\end{enumerate}
\end{proof}

\begin{remark}
    The above theorem generalizes the results concerning the left variable inclusion companion of a logic in Lemma \ref{lem:Sl-S,Sre-S}.
\end{remark}

\begin{theorem}\label{thm:DoubleRelComp2Logics}
    Suppose $\mathcal{S}_1=\langle\lang,\vdash_1\rangle,\mathcal{S}_2=\langle\lang,\vdash_2\rangle$ are logical structures such that $\vdash_1\,\subseteq\,\vdash_2$, and $\varrho,\sigma\subseteq\pow(\lang)\times\lang$ such that $\varrho\subseteq\sigma$. Then, $\vdash_1^\varrho\,\subseteq\,\vdash_2^\sigma$ and $\vdash_1^{p\varrho}\,\subseteq\,\vdash_2^{p\sigma}$.
\end{theorem}

\begin{proof}
    Suppose $\Gamma\cup\{\alpha\}\subseteq\lang$ such that $\Gamma\vdash_1^\varrho\alpha$. Then, there exists $\Delta\subseteq\Gamma$ such that $(\Delta,\alpha)\in\varrho$ and $\Delta\vdash_1\alpha$. Since $\varrho\subseteq\sigma$, $(\Delta,\alpha)\in\sigma$ and as $\vdash_1\,\subseteq\,\vdash_2$, $\Delta\vdash_2\alpha$. Thus, $\Gamma\vdash_2^\sigma\alpha$. Hence, $\vdash_1^\varrho\,\subseteq\,\vdash_2^\sigma$.

    It can be proved that $\vdash_1^{p\varrho}\,\subseteq\,\vdash_2^{p\sigma}$ with similar arguments.
\end{proof}

\begin{corollary}\label{cor:SingleRel2Logics}
    The following are some immediate observations from the above theorem.
    \begin{enumerate}[label=(\roman*)]
        \item Suppose $\mathcal{S}_1=\langle\lang,\vdash_1\rangle, \mathcal{S}_2=\langle\lang,\vdash_2\rangle$ are logical structures such that $\vdash_1\,\subseteq\,\vdash_2$ and $\varrho\subseteq\pow(\lang)\times\lang$. Then, $\vdash_1^\varrho\,\subseteq\,\vdash_2^\varrho$ and $\vdash_1^{p\varrho}\,\subseteq\,\vdash_2^{p\varrho}$.
        \item Suppose $\mathcal{S}=\langle\lang,\vdash\rangle$ is a logical structure and $\varrho,\sigma\subseteq\pow(\lang)\times\lang$. Then $\vdash^\varrho,\vdash^\sigma\,\subseteq\,\vdash^{\varrho\,\cup\,\sigma}$ and $\vdash^{\varrho\,\cap\,\sigma}\,\subseteq\,\vdash^\varrho,\vdash^\sigma$.
    \end{enumerate}
\end{corollary}

\begin{theorem}
    Suppose $\mathcal{S}=\langle\lang,\vdash\rangle$ is a logical structure and $\varrho,\sigma\subseteq\pow(\lang)\times\lang$.
    \begin{enumerate}[label=(\roman*)]
        \item $(\vdash^\varrho)^\sigma\,\subseteq\,\vdash^\varrho$. The equality holds if $\varrho\subseteq\sigma$.
        \item If $\mathcal{S}$ is monotonic, then $(\vdash^\varrho)^\sigma\,\subseteq\,\vdash^\sigma$.
        \item If $\varrho\subseteq\sigma$, then $(\vdash^\varrho)^\sigma\,\subseteq\,\vdash^\sigma$.
        \item If $\vdash^\sigma\,\subseteq\,\vdash^\varrho$, then $\vdash^\sigma\,\subseteq\,(\vdash^\varrho)^\sigma$.
        \item If $\varrho\subseteq\sigma$, then $\vdash^\varrho\,=\,\vdash^\sigma$ iff $(\vdash^\varrho)^\sigma\,=\,\vdash^\sigma$.
    \end{enumerate}
    Analogous statements hold for pure relational companions of $\mathcal{S}$.
\end{theorem}

\begin{proof}
    \begin{enumerate}[label=(\roman*)]
        \item The $\varrho$-companion of $\mathcal{S}$, $\mathcal{S}^\varrho$ is monotonic, by Theorem \ref{thm:SingleRelComp}(i). So, by Theorem \ref{thm:SingleRelComp}(ii), $(\vdash^\varrho)^\sigma\,\subseteq\,\vdash^\varrho$.

        Now, suppose $\varrho\subseteq\sigma$ and $\Gamma\cup\{\alpha\}\subseteq\lang$ such that $\Gamma\vdash^\varrho\alpha$. Then, there exists $\Delta\subseteq\Gamma$ such that $(\Delta,\alpha)\in\varrho$ and $\Delta\vdash\alpha$. Clearly, $\Delta\vdash^\varrho\alpha$. Since $\varrho\subseteq\sigma$, $(\Delta,\alpha)\in\sigma$. Thus, $\Gamma(\vdash^\varrho)^\sigma\alpha$. So, $\vdash^\sigma\,\subseteq\,(\vdash^\varrho)^\sigma$. Hence, $(\vdash^\varrho)^\sigma\,=\,\vdash^\varrho$.

        \item Suppose $\mathcal{S}$ is monotonic. Then, by Theorem \ref{thm:SingleRelComp}(ii), $\vdash^\varrho\,\subseteq\,\vdash$. Hence, by Corollary \ref{cor:SingleRel2Logics}(i), $(\vdash^\varrho)^\sigma\,\subseteq\,\vdash^\sigma$.

        \item Suppose $\varrho\subseteq\sigma$ and $\Gamma\cup\{\alpha\}\subseteq\lang$ such that $\Gamma(\vdash^\varrho)^\sigma\alpha$. Then, there exists $\Delta\subseteq\Gamma$ such that $(\Delta,\alpha)\in\sigma$ and $\Delta\vdash^\varrho\alpha$. So, there exists $\Delta^\prime\subseteq\Delta$ such that $(\Delta^\prime,\alpha)\in\varrho$ and $\Delta^\prime\vdash\alpha$. Now, $\Delta^\prime\subseteq\Gamma$, and as $\varrho\subseteq\sigma$, $(\Delta^\prime,\alpha)\in\sigma$. So, $\Gamma\vdash^\sigma\alpha$. Thus, $(\vdash^\varrho)^\sigma\,\subseteq\,\vdash^\sigma$.

        \item Suppose $\vdash^\sigma\,\subseteq\,\vdash^\varrho$ and $\Gamma\cup\{\alpha\}\subseteq\lang$ such that $\Gamma\vdash^\sigma\alpha$. Then, there exists $\Delta\subseteq\Gamma$ such that $(\Delta,\alpha)\in\sigma$ and $\Delta\vdash\alpha$. Clearly, $\Delta\vdash^\sigma\alpha$. Now, as $\vdash^\sigma\,\subseteq\,\vdash^\varrho$, $\Delta\vdash^\varrho\alpha$ as well. Thus, $\Gamma(\vdash^\varrho)^\sigma\alpha$. Hence, $\vdash^\sigma\,\subseteq\,(\vdash^\varrho)^\sigma$.

        \item Suppose $\varrho\subseteq\sigma$. Moreover, suppose $\vdash^\varrho\,=\,\vdash^\sigma$. Then, using parts (iii) and (iv) above, we have $(\vdash^\varrho)^\sigma\,=\,\vdash^\sigma$. 
        
        Conversely, suppose $(\vdash^\varrho)^\sigma\,=\,\vdash^\sigma$. Let $\Gamma\cup\{\alpha\}\subseteq\lang$ such that $\Gamma\vdash^\sigma\alpha$. So, $\Gamma(\vdash^\varrho)^\sigma\alpha$. Then, there exists $\Delta\subseteq\Gamma$ such that $(\Delta,\alpha)\in\sigma$ and $\Delta\vdash^\varrho\alpha$. This implies that there exists $\Delta^\prime\subseteq\Delta$ such that $(\Delta^\prime,\alpha)\in\varrho$ and $\Delta^\prime\vdash\alpha$. Since $\Delta^\prime\subseteq\Gamma$, $\Gamma\vdash^\varrho\alpha$. Thus. $\vdash^\sigma\,\subseteq\,\vdash^\varrho$. Now, as $\varrho\subseteq\sigma$, $\vdash^\varrho\,\subseteq\,\vdash^\sigma$, by Theorem \ref{thm:DoubleRelComp2Logics}. Hence, $\vdash^\varrho\,=\,\vdash^\sigma$.
    \end{enumerate}

    The proofs for the pure relational companions of $\mathcal{S}$ can be constructed with similar arguments.
\end{proof}

\begin{definition}
    Suppose $A$ is a set and $\varrho\subseteq\pow(A)\times A$. Then, $\varrho$ is said to be \emph{downward directed} if for any $\Delta\cup\{\alpha\}\subseteq\lang$, $(\Delta,\alpha)\in\varrho$ implies that $(\Delta^\prime,\alpha)\in\varrho$ for all $\Delta^\prime\subseteq\Delta$.
\end{definition}

\begin{theorem}
    Suppose $\mathcal{S}=\langle\lang,\vdash\rangle$ is a logical structure and $\varrho,\sigma\subseteq\pow(\lang)\times\lang$. If $\sigma$ is downward directed, then $(\vdash^\varrho)^\sigma\,\subseteq\,(\vdash^\sigma)^\varrho$. Hence, if $\varrho,\sigma$ are both downward directed, then $(\vdash^\varrho)^\sigma\,=\,(\vdash^\sigma)^\varrho$.
\end{theorem}

\begin{proof}
    Suppose $\sigma$ is downward directed. Let $\Gamma\cup\{\alpha\}\subseteq\lang$ such that $\Gamma(\vdash^\varrho)^\sigma\alpha$. Then, there exists $\Delta\subseteq\Gamma$ such that $(\Delta,\alpha)\in\sigma$ and $\Delta\vdash^\varrho\alpha$. This implies that there exists $\Delta^\prime\subseteq\Delta$ such that $(\Delta^\prime,\alpha)\in\varrho$ and $\Delta^\prime\vdash\alpha$. Now, as $\sigma$ is downward directed, $(\Delta^\prime,\alpha)\in\sigma$. Then, $\Delta^\prime\vdash^\sigma\alpha$. Now, since $(\Delta^\prime,\alpha)\in\varrho$ and $\Delta^\prime\subseteq\Gamma$, this implies that $\Gamma(\vdash^\sigma)^\varrho\alpha$. Hence, $(\vdash^\varrho)^\sigma\,\subseteq\,(\vdash^\sigma)^\varrho$.

    Thus, if $\varrho$ is also downward directed, then $(\vdash^\sigma)^\varrho\,\subseteq\,(\vdash^\varrho)^\sigma$, and hence, in that case, $(\vdash^\varrho)^\sigma\,=\,(\vdash^\sigma)^\varrho$.
\end{proof}

It is not hard to see that, for any logic $\mathcal{S}=\langle\lang,\vdash\rangle$, $\vdash\alpha$ iff $\vdash^l\alpha$ for all $\alpha\in\lang$. The following theorem generalizes this to relational companions of logical structures.

\begin{theorem}
    Suppose $\mathcal{S}=\langle\lang,\vdash\rangle$ is a logical structure and $\varrho\subseteq\pow(\lang)\times\lang$ such that $(\emptyset,\alpha)\in\varrho$ for all $\alpha\in\lang$. Then, for any $\alpha\in\lang$, $\emptyset\vdash\alpha$ (written as $\vdash\alpha$) iff $\emptyset\vdash^\varrho\alpha$ (written as $\vdash^\varrho\alpha$).
\end{theorem}

\begin{proof}
    Suppose $\alpha\in\lang$ such that $\vdash\alpha$. Then, as $(\emptyset,\alpha)\in\varrho$, $\vdash^\varrho\alpha$. Conversely, suppose $\vdash^\varrho\alpha$. Then, there exists $\Delta\subseteq\emptyset$ such that $(\Delta,\alpha)\in\varrho$ and $\Delta\vdash\alpha$. Clearly, $\Delta=\emptyset$. Thus, $\vdash\alpha$. 
\end{proof}

As discussed above, the left, pure left, and pure right variable inclusion companions of a logic are paraconsistent (see \cite{BasuRoy2023} for a detailed study on this). The following theorem generalizes this to certain relational companions of logical structures.

\begin{definition}
    Suppose $A, B$ are sets and $\varrho\subseteq A\times B$. Then, $\varrho$ is said to have \emph{finite reach} if for every $a\in A$, the set $\varrho_a=\{b\in B\mid\,(a,b)\in\varrho\}$ is finite.
\end{definition}

\begin{theorem}\label{thm:FinReach->NoFinTriv}
    Suppose $\mathcal{S}=\langle\lang,\vdash\rangle$ is a logical structure such that $\lang$ is infinite and $\varrho\subseteq\pow(\lang)\times\lang$ has finite reach. Then, for every finite $\Gamma\subseteq\lang$, there exists $\beta\in\lang$ such that $\Gamma\not\vdash^\varrho\beta$, i.e., every finite $\Gamma\subseteq\lang$ is nontrivial in $\mathcal{S}^\varrho$. The same is true for the pure $\varrho$-companion of $\mathcal{S}$.
\end{theorem}

\begin{proof}
    Suppose $\varrho$ has finite reach. Let $\Gamma$ be a finite subset of $\lang$. So, $\pow(\Gamma)$ is finite. Now, as $\varrho$ has finite reach, $\varrho_\Delta=\{\alpha\in\lang\mid\,(\Delta,\alpha)\in\varrho\}$ is finite for every $\Delta\in\pow(\Gamma)$. Thus, $\displaystyle\bigcup_{\Delta\in\pow(\Gamma)}\varrho_\Delta$, being a finite union of finite sets, is finite. Let $\beta\in\lang\setminus\displaystyle\bigcup_{\Delta\in\pow(\Gamma)}\varrho_\Delta$ (such a $\beta$ exists since $\lang$ is infinite). Then, $(\Delta,\beta)\notin\varrho$ for every $\Delta\subseteq\Gamma$. Thus, $\Gamma\not\vdash^\varrho\beta$. Hence, every finite subset of $\lang$ is non-trivial.

    An almost identical argument proves the same result for $\mathcal{S}^{p\varrho}=\langle\lang,\vdash^{p\varrho}\rangle$.
\end{proof}

\begin{corollary}
    Suppose $\mathcal{S}=\langle\lang,\vdash\rangle$ is a logic (in the usual sense) such that there is a unary (negation) operator $\neg$ in the signature. Moreover, let $\varrho\subseteq\pow(\lang)\times\lang$ be a relation with finite reach. Then, $\neg$-ECQ fails in $\mathcal{S}^\varrho$ and $\mathcal{S}^{p\varrho}$. Thus, $\mathcal{S}^\varrho$ and $\mathcal{S}^{p\varrho}$ are paraconsistent with respect to $\neg$.
\end{corollary}

\begin{definition}
    A logical structure $\mathcal{S}=\langle\lang,\vdash\rangle$ is said to be \emph{finitely trivializable} if there exists a finite $\Gamma\subseteq\lang$ such that $\Gamma\vdash\alpha$ for all $\alpha\in\lang$, i.e., there is a finite trivial subset of $\lang$.
\end{definition}

\begin{remark}\label{rem:FinReach->NoFinTriv}
    Suppose $\mathcal{S}=\langle\lang,\vdash\rangle$ is a logical structure such that $\lang$ is infinite and $\varrho\subseteq\pow(\lang)\times\lang$ is a relation with finite reach. Then, by Theorem \ref{thm:FinReach->NoFinTriv}, $\mathcal{S}^\varrho$ and $\mathcal{S}^{p\varrho}$ are not finitely trivializable.
\end{remark}

Suppose $\mathcal{S}=\langle\lang,\vdash\rangle$ is a logical structure. The generalized principle of explosion (gECQ) was introduced in \cite{BasuRoy2022} as follows. For every $\alpha\in\lang$, there exists $\beta\in\lang$ such that $\{\alpha,\beta\}\vdash\gamma$ for all $\gamma\in\lang$. A logic or logical structure in which gECQ fails is called NF-paraconsistent (NF stands for Negation-Free). 

spECQ (where sp stands for set-point), a principle of explosion introduced in \cite{BasuRoy2023}, can be described as follows. For each $\Gamma\subsetneq\lang$, there exists $\alpha\in\lang$ such that $\Gamma\cup\{\alpha\}\subsetneq\lang$ and $\Gamma\cup\{\alpha\}\vdash\beta$ for all $\beta\in\lang$.

It is easy to see that, if $\mathcal{S}=\langle\lang,\vdash\rangle$ is a logical structure that is not finitely trivializable, then gECQ and spECQ fail in it. This is also discussed in \cite{BasuRoy2023}.

\begin{corollary}
    Suppose $\mathcal{S}=\langle\lang,\vdash\rangle$ is a logical structure such that $\lang$ is infinite and $\varrho\subseteq\pow(\lang)\times\lang$ is a relation with finite reach. Then, by Remark \ref{rem:FinReach->NoFinTriv}, gECQ and spECQ fail in $\mathcal{S}^\varrho$ and $\mathcal{S}^{p\varrho}$.
\end{corollary}

\subsection{Hilbert-type logical structures and their restrictions}
In this subsection, we first generalize Hilbert-style logics to logical structures and then discuss `restricted' companions of these as generalizations of the restricted rules companions of Hilbert-style logics discussed in Section 2. We first note that any axiom in a Hilbert-style logic can also be regarded as a rule with an empty set of hypotheses. Thus, it is sufficient to deal with Hilbert-style logics induced by only a set of rules. Secondly, for the restricted rules companion of a logic, we used the variable inclusion restriction only to restrict all the rules of inference. This can be generalized to a situation where different rules are restricted in different ways.

\begin{definition}
    Suppose $\lang$ is a set and $\mathcal{R}\subseteq\pow(\lang)\times\lang$. Let $\vdash\,\subseteq\,\pow(\lang)\times\lang$ be defined as follows. For any $\Gamma\cup\{\alpha\}\subseteq\lang$, $\Gamma\vdash\alpha$ if there exists a finite sequence $(\beta_0,\ldots,\beta_n)$ of elements of $\lang$ with $\beta_n=\alpha$, and for each $0\le i\le n$, either $\beta_i\in\Gamma$, or there exists $\Gamma^\prime\subseteq\{\beta_0\ldots,\beta_{i-1}\}$ such that $(\Gamma^\prime,\beta_i)\in\mathcal{R}$. Then $\mathcal{S}=\langle\lang,\vdash\rangle$ is called the \emph{Hilbert-type logical structure induced by $\mathcal{R}$}.

    A \emph{Hilbert-type logical structure} $\langle\lang,\vdash\rangle$ is any logical structure that is induced by some $\mathcal{R}\subseteq\pow(\lang)\times\lang$.
\end{definition}

\begin{remark}\label{rem:Hilbert-Tarski}
    It is clear from the above definition that a Hilbert-type logical structure $\mathcal{S}=\langle\lang,\vdash\rangle$ is finitary and Tarski-type, i.e., reflexive, monotonic, and transitive (see \cite{RoyBasuChakraborty2023} for the definitions).
\end{remark}

\begin{definition}
    Suppose $\mathcal{S}=\langle\lang,\vdash\rangle$ is a Hilbert-type logical structure induced by $\mathcal{R}\subseteq\pow(\lang)\times\lang$ and $\Pi$ is a collection of relations from $\pow(\lang)$ to $\lang$, i.e., for any $\sigma\in\Pi$, $\sigma\subseteq\pow(\lang)\times\lang$. Let $\mathcal{R}^\Pi=\{(\Gamma,\alpha)\in\mathcal{R}\mid\,\hbox{there exists }\sigma\in\Pi\hbox{ such that }(\Gamma,\alpha)\in\sigma\}$. Then, the logical structure induced by $\mathcal{R}^\Pi$, denoted by $\mathcal{S}^\Pi=\langle\lang,\vdash^\Pi\rangle$, is called the \emph{$\Pi$-restricted companion} of $\mathcal{S}$.
\end{definition}

\begin{remark}
     Suppose $\mathcal{S}=\langle\lang,\vdash\rangle$ is a Hilbert-style logic with $\mathcal{R}$ as the set of rules of inference. Let $\Pi=\{\sigma\}$, where $\sigma\subseteq\pow(\lang)\times\lang$ is defined by $(\Gamma,\alpha)\in\sigma$ iff $\var(\Gamma)\subseteq\var(\alpha)$. Then, $\mathcal{R}^\Pi=\{(\Gamma,\alpha)\in\mathcal{R}\mid\,\var(\Gamma)\subseteq\var(\alpha)\}$, and hence $\mathcal{S}^\Pi=\langle\lang,\vdash^\Pi\rangle$, the $\Pi$-restricted companion of $\mathcal{S}$, is the restricted rules companion  of $\mathcal{S}$, $\mathcal{S}^{re}=\langle\lang,\vdash^{re}\rangle$.
\end{remark}

We can also see a $\Pi$-restricted companion of a Hilbert-type logical structure as a relational companion of a logical structure as follows.

Suppose $\mathcal{S}=\langle\lang,\vdash\rangle$ is a logical structure induced by $\mathcal{R}\subseteq\pow(\lang)\times\lang$ and $\Pi$ is a collection of relations from $\pow(\lang)$ to $\lang$. Then, $\mathcal{S}^\Pi=\langle\lang,\vdash^\Pi\rangle$ is the Hilbert-type logical structure induced by $\mathcal{R}^\Pi$, where $\mathcal{R}^\Pi$ is as described in the above definition. Let $\varrho\subseteq\pow(\lang)\times\lang$ be defined as follows. $(\Delta,\alpha)\in\varrho$ iff $\Delta\vdash^\Pi\alpha$. Then, $\mathcal{S}^\varrho=\langle\lang,\vdash^\varrho\rangle$ is the $\varrho$-companion of $\mathcal{S}$.

Now, suppose $\Gamma\cup\{\alpha\}\subseteq\lang$ and $\Gamma\vdash^\Pi\alpha$. So, there exists a finite sequence $(\beta_0,\ldots,\beta_n)$ of elements of $\lang$ with $\beta_n=\alpha$, and for each $0\le i\le n$, either $\beta_i\in\Gamma$, or there exists $\Gamma^\prime\subseteq\{\beta_0\ldots,\beta_{i-1}\}$ such that $(\Gamma^\prime,\beta_i)\in\mathcal{R}^\Pi$. Let $\Delta=\Gamma\cap\{\beta_0,\ldots,\beta_n\}$. Clearly, $\Delta\vdash^\Pi\alpha$, i.e., $(\Delta,\alpha)\in\varrho$. So, $\Gamma\vdash^\varrho\alpha$. Conversely, suppose $\Gamma\vdash^\varrho\alpha$. Then, there exists a $\Delta\subseteq\Gamma$ such that $(\Delta,\alpha)\in\varrho$, i.e., $\Delta\vdash^\Pi\alpha$ and $\Delta\vdash\alpha$. Now, since $\mathcal{S}^\Pi$ is a Hilbert-type logical structure, and hence, monotonic by Remark \ref{rem:Hilbert-Tarski}, $\Delta\vdash^\Pi\alpha$ implies $\Gamma\vdash^\Pi\alpha$. Hence, $\vdash^\Pi\,=\,\vdash^\varrho$.

Thus, the restricted rules companions of a logic can also be seen as relational companions of the logic.

\section{Conclusions and future directions}
In this article, we have proposed the relational companions of logical structures as generalizations of the variable inclusion companions of logics and the restricted rules companions of Hilbert-style logics. More properties of these companion logics, especially those of the $\Pi$-restricted companions of Hilbert-type logics, could be investigated further. 

Another, perhaps interesting, observation is the possibility of using a logical structure $\mathcal{S}_1=\langle\lang,\vdash_1\rangle$ to define a companion to another logical structure $\mathcal{S}_2=\langle\lang,\vdash_2\rangle$ in the following way. Since $\vdash_1\subseteq\pow(\lang)\times\lang$ is also a relation from $\pow(\lang)$ to $\lang$, we can define the $\vdash_1$-companion of $\mathcal{S}_2$. This is already the case in the way we have described a $\Pi$-restricted companion of a Hilbert-type logical structure as a relational companion at the end of the previous section. This line of investigation can be interesting from the perspective of combining logical structures as well, since this is, in a way, merging the two relations $\vdash_1$ and $\vdash_2$.

\bibliographystyle{amsplain}
\bibliography{CompLog}

\end{document}